\documentclass{article}

\makeatletter

\@addtoreset{equation}{section}
\makeatother

\input xy
\xyoption{all}

\usepackage[cp1252]{inputenc}
\usepackage[OT1]{fontenc}
\usepackage[english]{babel}

\usepackage{amssymb, amsmath, amsthm,
amscd,
graphicx,
cite,
comment
}

\def\R{{\mathbb R}}
\def\N{{\mathbb N}}
\def\Z{{\mathbb Z}}

\def\CA{\mathcal{A}}

\def\ds{\displaystyle}

\def\b{\beta}

\def\D{\Delta}

\def\eps{\varepsilon}
\def\f{\varphi}

\def\t{\theta}

\def\lam{\lambda}

\def\dth{\dot \theta}

\def\dlam{\dot \lam}
\def\dx{\dot x}
\def\dy{\dot y}

\def\dq{\dot q}
\def\dh{\dot h}

\def\tcut{t_{\operatorname{cut}}}
\def\tconj{t_1^{\operatorname{conj}}}
\def\tt{\mathbf t}

\newcommand{\am}{\operatorname{am}\nolimits}
\newcommand{\sn}{\operatorname{sn}\nolimits}
\newcommand{\cn}{\operatorname{cn}\nolimits}
\newcommand{\dn}{\operatorname{dn}\nolimits}
\newcommand{\sgn}{\operatorname{sgn}\nolimits}
\newcommand{\const}{\operatorname{const}\nolimits}
\newcommand{\Id}{\operatorname{Id}\nolimits}

\newcommand{\E}{\operatorname{E}\nolimits}

\renewcommand{\Vec}{\operatorname{Vec}\nolimits}
\newcommand{\Exp}{\operatorname{Exp}\nolimits}

\newcommand{\cl}{\operatorname{cl}\nolimits}

\def\vH{\vec H}
\def\vP{\vec P}

\newcommand{\der}[2]{\frac{d \, #1}{d\, #2} }
\newcommand{\pder}[2]{\frac{\partial \, #1}{\partial \, #2} }
\newcommand{\be}[1]{\begin{equation}\label{#1}}
\newcommand{\ee}{\end{equation}}
\newcommand{\ddef}[1]{{#1}}
\newcommand{\map}[3]{#1 \, : \, #2 \to #3}
\newcommand{\mapto}[3]{#1 \, : \, #2 \mapsto #3}
\newcommand{\vect}[1]{\left( \begin{array}{c} #1 \end{array} \right)}
\newcommand{\eq}[1]{$(\protect\ref{#1})$}
\newcommand{\restr}[2]{\left. #1 \right|_{#2}}

\def\tlam{\widetilde{\lambda}}
\def\tM{\widetilde{M}}
\def\tN{\widetilde{N}}

\def\tq{\widetilde{q}}

\def\hN{\widehat{N}}

\def\btau{\bar{\tau}}

\def\bp{\bar{p}}
\def\bk{\bar{k}}
\def\bq{\bar{q}}

\def\bt{\bar{t}}
\def\bu{\bar{u}}
\def\bp{\bar{p}}

\def\blam{\bar{\lam}}
\def\oR{\overline{\R}}

\def\ts{\,{\sn \tau}\,}
\def\tc{\,{\cn \tau}\,}
\def\td{\,{\dn \tau}\,}
\def\ss{\,{\sn  p}\,}
\def\cc{\,{\cn  p}\,}
\def\dd{\,{\dn  p}\,}
\def\tsp{\,{\sn^2 \tau}\,}

\def\ssp{\,{\sn^2  p}\,}
\def\ccp{\,{\cn^2  p}\,}
\def\ddp{\,{\dn^2  p}\,}

\def\then{\quad\Rightarrow\quad}
\def\iff{\quad\Leftrightarrow\quad}

\newcommand{\twofiglabel}[6]
{
\begin{figure}[htbp]
\includegraphics[width=0.47\textwidth]{#1}
\hfill
\includegraphics[width=0.47\textwidth]{#4}
\\
\parbox[t]{0.45\textwidth}{\caption{#2}\label{#3}}
\hfill
\parbox[t]{0.45\textwidth}{\caption{#5}\label{#6}}
\end{figure}
}

\newcommand{\onefiglabel}[3]
{
\begin{figure}[htbp]
\begin{center}
\includegraphics[width=0.47\textwidth]{#1}
\\
\parbox[t]{0.45\textwidth}{\caption{#2}\label{#3}}
\end{center}
\end{figure}
}

\newtheorem{theorem}{Theorem}[section]
\newtheorem{lemma}{Lemma}[section]
\newtheorem{corollary}{Corollary}[section]
\newtheorem{proposition}{Proposition}[section]

\theoremstyle{remark}

\title{Exponential mapping in Euler's elastic problem%
\footnote{Work supported by the Russian Foundation for Basic Research, project
No.~12-01-00913-a.}
}

\author{
Yu. L. Sachkov,  E. F. Sachkova\\
Program Systems Institute \\
Russian Academy of Sciences \\
Pereslavl-Zalessky 152020 Russia \\
E-mail: sachkov@sys.botik.ru}

\date{\today}

\begin{document}

\maketitle

\begin{abstract}
The classical Euler's problem on optimal configurations of elastic rod in the plane with fixed endpoints and tangents at the endpoints is considered. 
The global structure of the exponential mapping that parameterises extremal trajectories is described. It is proved that open domains cut out by Maxwell strata in the preimage and image of the exponential mapping  are mapped diffeomorphically. As a consequence, computation of globally optimal elasticae with given boundary conditions is reduced to solving systems of algebraic equations having unique solutions in the open domains. For certain special boundary conditions, optimal elasticae are presented.



\end{abstract}

\section{Introduction}\label{sec:intro}
This work is devoted to the study of the following problem considered by Leonhard Euler~\cite{euler,love}. Given an elastic rod in the plane with fixed endpoints and tangents at the endpoints, one should determine possible profiles of the rod under the given boundary conditions. Euler's problem can be stated as the following optimal control problem:
\begin{align}
&\dx = \cos \t, \label{sys1} \\
&\dy = \sin \t, \label{sys2} \\
&\dot \t = u, \label{sys3} \\
&q=(x,y,\t) \in M= \R^2_{x,y} \times S^1_{\t}, \qquad u \in \R,  \label{sys4} \\
&q(0) = q_0 = (x_0, y_0, \t_0), \qquad
q(t_1) = q_1 = (x_1, y_1, \t_1),
\qquad t_1 \text{ fixed}, \label{sys5}\\
&J = \frac{1}{2} \int_0^{t_1} u^2(t) \, dt \to \min. \label{sys6}
\end{align}
where the integral $J$ evaluates the elastic energy of the rod $(x(t), y(t))$.

This paper is an immediate  continuation of the previous works~\cite{el_max, el_conj}, which contained the following material: history of the problem, description of attainable set, proof of existence and boundedness of optimal controls, parameterisation of extremals by Jacobi's functions, description of discrete symmetries and the corresponding Maxwell points, bounds on cut time and conjugate time. In this work we widely use the notation, definitions, and results of work~\cite{el_max, el_conj}.

The upper bound of cut time on extremal trajectories via Maxwell points obtained in~\cite{el_max, el_conj} allows to get rid of necessarily non-optimal candidates in the search of optimal trajectories. There were left open questions on the number of remaining candidates for optimal trajectories, and on the number of optimal trajectories with given boundary conditions. This paper answers these questions. We show that for generic boundary conditions there remain two candidates for optimal trajectories that satisfy the upper bound on cut time obtained in~\cite{el_max, el_conj}. We prove that the search of these two candidates can be reduced to solving systems of algebraic equations having unique solutions in certain domains. After these candidates are computed, it remains to compare their costs and find the trajectory with the less cost.
For generic boundary conditions (where the costs of two candidates differ one from another) there is a unique optimal trajectory. If the two candidates have the same cost, then there are two optimal trajectories with the given boundary conditions.

Further, we consider several families of special boundary conditions and specify optimal trajectories for them. We present examples of boundary conditions with 1, 2, and 4 optimal trajectories (we believe no other numbers of optimal trajectories  occur).

The structure of this paper is as follows. In Sec.~\ref{sec:known} we recall some necessary results of the previous works~\cite{el_max, el_conj}. In particular, we recall definition of the exponential mapping that parameterises endpoints of extremal trajectories at a given instant of time. In Sec.~\ref{sec:decompos} we introduce decompositions of preimage and image of the exponential mapping into certain open domains and their boundary, and prove some topological properties of this decomposition. In Sec.~\ref{sec:diffeo} we show that restriction of the exponential mapping to these open domains is a diffeomorphism, which guarantees unique solvability of algebraic equations for candidates for optimal trajectories. In Sec.~\ref{sec:boundary} we describe the action of the exponential mapping on the boundary of the diffeomorphic domains. Finally, in Sec.~\ref{sec:various} we describe optimal trajectories for various boundary conditions. 

\section{Previous results on Euler's problem}\label{sec:known}
In this section we recall some necessary results of the previous works~\cite{el_max, el_conj}.

By virtue of  parallel translations and rotations  in the plane $\R^2_{x,y}$ (problem~\eq{sys1}--\eq{sys6} is left-invariant on the group of motions of the plane), we can assume that 
\be{q0=0}
q_0 = (x_0, y_0, \t_0) = (0, 0, 0).
\ee
Moreover, due to the following one-parameter group of symmetries (dilations in the plane $\R^2_{x,y}$):
\be{group1}
(x,y,\t,t,u,t_1,J) \mapsto (\tilde x, \tilde y, \tilde \t, \tilde t, \tilde u, \tilde t_1, \tilde J) = (e^sx, e^sy, \t, e^st, e^{-s}u, e^st_1, e^{-s}J),
\ee
 we can assume that the terminal time (length of elastica) is $t_1 = 1$.
 
 Attainable set of system~\eq{sys1}--\eq{sys4} from the point $q_0 =  (0, 0, 0)$ for time $t_1 = 1$ is
\begin{align*}
\CA =
\{ (x,y,\t) \in M 
&\mid
x^2 + y^2 < 1
\text{ or }
(x,y,\t) = (1, 0, 0) \},
\end{align*}
see Th.~4.1~\cite{el_max}. For any $q_1 \in \CA$, there exists an optimal trajectory that satisfies Pontryagin maximum principle (Th.~5.3~\cite{el_max}). 

Denote the vector fields in the right-hand side of system~\eq{sys1}--\eq{sys3} and their Lie bracket:
$$
X_1 = \cos \t \pder{}{x} + \sin \t \pder{}{y},
\qquad X_2 = \pder{}{\t}, \qquad
 X_3 = [X_1, X_2] = \sin \t \pder{}{x} - \cos \t \pder{}{y}.
$$
Consider the corresponding Hamiltonians, linear on fibers in the cotangent bundle $T^*M$:
$$
h_i(\lam) = \langle \lam, X_i\rangle,
\qquad \lam \in T^*M, \quad i = 1, 2, 3.
$$
The normal Hamiltonian of Pontryagin maximum principle for the elastic problem is
$\ds H = h_1 + \frac12 h_2^2$,
and the corresponding \ddef{normal Hamiltonian system of PMP} reads
\be{dh1-h2h3}
\dlam = \vH(\lam) \iff
\begin{cases}
\dh_1 = -h_2h_3, \\
\dh_2 = h_3, \\
\dh_3 = h_1 h_2, \\
\dq = X_1 + h_2 X_2.
\end{cases}
\ee
The vertical subsystem of system~\eq{dh1-h2h3} has an obvious integral:
$$
h_1^2 + h_3^2 \equiv r^2 = \const \geq 0,
$$
and it is natural to introduce the coordinates
$$
h_1 = - r \cos \b, \qquad h_3 = -r \sin \b, \qquad h_2 = c.
$$
Then the normal Hamiltonian system~\eq{dh1-h2h3} takes the following form:
\be{dadrdh2}
\begin{cases}
\dot \b = c, \\
\dot c = - r \sin \b, \\
\dot r = 0, \\
\dx = \cos \t, \\
\dy = \sin \t, \\
\dth = c.
\end{cases}
\ee
The total energy of the equation of pendulum 
\be{pend}
\dot \b = c, \qquad
\dot c = - r \sin \b, \qquad
\dot r = 0
\ee
 is
\be{E_r}
E = \frac{c^2}{2} - r \cos \b \in [-r, + \infty).
\ee
The normal Hamiltonian system~\eq{dadrdh2} was integrated in~\cite{el_max}.

The time $t$ \ddef{exponential mapping} for the problem is defined as follows:
$$
\map{\Exp_{t}}{N = T_{q_0}^*M}{M},
\qquad
\Exp_{t}(\lam_0) = \pi \circ e^{t \vH}(\lam_0) = q(t).
$$
We will denote the exponential mapping for time $t_1 = 1$ as $\Exp$. 

Preimage $N$ of the exponential mapping admits the following decomposition into disjoint subsets:
\begin{align}
&N = \bigsqcup_{i=1}^7 N_i, \label{N_decomp} \\
&N_1 = \{ \lam \in N \mid r \neq 0, \ E \in (-r, r)\}, \label{N1} \\
&N_2 = \{ \lam \in N \mid r \neq 0, \ E \in (r, + \infty)\} = N_2^+ \sqcup N_2^-, \label{N2} \\
&N_3 = \{ \lam \in N \mid r \neq 0, \ E  = r, \ \b \neq \pi \} = N_3^+ \sqcup N_3^-, \label{N3}  \\
&N_4 = \{ \lam \in N \mid r \neq 0, \ E =-r \}, \label{N4} \\
&N_5 = \{ \lam \in N \mid r \neq 0, \ E =r, \ \b = \pi \}, \label{N5} \\
&N_6 = \{ \lam \in N \mid r = 0, \ c \neq 0 \} = N_6^+ \sqcup N_6^-,\label{N6} \\
&N_7 = \{ \lam \in N \mid r = c =  0\}, \label{N7} \\
&N_i^{\pm} = N_i \cap \{ \lam \in N \mid \sgn c = \pm 1 \}, \qquad i = 2, \ 3, \ 6. \label{Ni+-}
\end{align}

In Sec.~7~\cite{el_max} were introduced elliptic coordinates $(\f, k, r)$ in the domain $N_1 \cup N_2 \cup N_3$ which rectify the flow of  pendulum~\eq{pend}:
\be{rectify}
\dot \f = 1, \qquad \dot k = \dot r = 0.
\ee
These coordinates have the following ranges:
\begin{align*}
&\lam = (\f, k, r) \in N_1 \then r > 0, \quad k \in (0, 1), \quad \f \in \R \pmod{ 4 K/\sqrt r}, \\ 
&\lam = (\f, k, r) \in N_2 \then r > 0, \quad k \in (0, 1), \quad \f \in \R \pmod{ 2 K k/\sqrt r}, \\ 
&\lam = (\f, k, r) \in N_3 \then r > 0, \quad k = 1, \quad \f \in \R,  
\end{align*}
where $K(k)$ is the complete elliptic integral of the first kind~\cite{whit_watson}.

Further, in~\cite{el_max} were introduced  coordinates $(p, \tau, k)$ in the domain $N_1 \cup N_2 \cup N_3$ as follows:
 \begin{align*}
&\lam   \in N_1 \then p = \frac{\sqrt r}{2} > 0, \quad \tau = \sqrt r \left(\f + \frac 12\right) \in \R \pmod{ 4 K}, \quad k \in (0, 1),\\  
&\lam   \in N_2 \then p = \frac{\sqrt r}{2 k} > 0, \quad \tau = \frac{\sqrt r}{k} \left(\f + \frac 12\right) \in \R \pmod{ 2 K}, \quad k \in (0, 1),\\
&\lam   \in N_3 \then p = \frac{\sqrt r}{2} > 0, \quad \tau = \sqrt r \left(\f + \frac 12\right) \in \R, \quad k =1.
\end{align*}

In~\cite{el_max, el_conj} was obtained the upper bound~\eq{tcutbound} of the cut time 
$$
\tcut = \sup \{t_1 > 0  \mid \text{ extremal trajectory } q(t) \text{ is optimal on } [0, t_1]\}
$$
in terms of the following function:
\begin{align}
&\map{\tt}{N}{(0, + \infty]}, \qquad \lam \mapsto \tt(\lam), \label{tt} \\
&\lam \in N_1 \then \tt = \frac{2}{\sqrt r} p_1(k), \label{ttN1} \\
&\qquad
p_1(k) = \min(2K(k), p_1^1(k)) =
\begin{cases}
2 K(k), & k \in (0, k_0] \\
p_1^1(k), &k \in [k_0, 1)
\end{cases}
\label{p1(k)N1}\\
&\lam \in N_2 \then \tt = \frac{2k}{\sqrt r} p_1(k), \qquad p_1(k) = K(k),
\label{ttN2} \\
&\lam \in N_6 \then \tt = \frac{2\pi}{|c|}, \nonumber\\
&\lam \in N_3\cup N_4\cup N_5\cup N_7 \then \tt = + \infty. \label{ttN3}
\end{align}
Here $p = p_1^1(k)$ is the first positive root of the equation
\be{f1=0}
f_1(p,k) = \ss \dd - (2 \E(p) - p) \cc = 0, \qquad p \in  (K, 3 K),
\ee
(see Propos.~11.6~\cite{el_max}), 
where $\ss$, $\cc$, $\dd$ are Jacobi's elliptic functions, $\E(p) = \int_0^p \dn^2 t \, dt$,
and $k_0$ is the unique root of the equation $2 E(k) - K(k) = 0$ (see Propos.~11.5~\cite{el_max}). Here and below $E(k)$ is the complete elliptic integral of the second kind~\cite{whit_watson}.

\section{Decompositions in preimage and image \\of exponential mapping}\label{sec:decompos}
Existence of optimal controls implies that
the mapping $\map{\Exp}{N}{\CA}$ is surjective.
Theorem~5.1~\cite{el_conj} states that
\be{tcutbound}
\forall \lam \in N \qquad \tcut(\lam) \leq \tt(\lam).
\ee
Thus for any $\lam \in N$ with $\tt(\lam) < 1$,  the extremal trajectory $q(t) = \Exp_t(\lam)$ is not optimal at the segment $t \in [0, 1]$. Consequently, for any $q_1 \in \CA$ there exists an optimal trajectory $\tq(t) = \Exp_t(\tlam)$, $\lam \in N$,  $t \in [0, 1]$, such that $q(1) = q_1$, so $\tt(\tlam) \geq 1$. Define the corresponding set
$$
\hN = \{ \lam \in N \mid \tt(\lam) \geq 1 \}.
$$
Then 
the mapping $\map{\Exp}{\hN}{\CA}$ is surjective.

\subsection{Definition of decomposition in preimage \\of exponential mapping}
Introduce the following decomposition of the set $\hN$:
\begin{align}
&\hN = \tN \sqcup N', \label{NhatNtilde}\\
&\tN = \{ \lam \in \cup_{i=1}^3 N_i \mid \tt(\lam) > 1, \ \tc \ts \neq 0\},  \nonumber\\
&N' = N'_{1-3} \sqcup N_4 \sqcup N_5 \sqcup \hN_6 \sqcup N_7, \label{N'}\\
&N'_{1-3} =  \{ \lam \in \cup_{i=1}^3 N_i \mid \tt(\lam) = 1 \text{ or } \tc \ts = 0\},  \nonumber\\
&\hN_6 = N_6 \cap \hN.    \nonumber
\end{align}
Moreover, the set $\tN$ naturally decomposes as follows:
\be{tNLi}
\tN = \bigsqcup_{i=1}^4 L_i,
\ee
with the sets $L_i$ defined by Table~\ref{tab:Li}.

\begin{table}[htbp]
$$
\begin{array}{|c|c|c|c|c|}
\hline
L_i & L_1 & L_2 & L_3 & L_4  \\
\hline
\lam & N_1 & N_1 & N_1   & N_1   \\
\tau & (0,K) & (K,2K) & (2K,3K) & (3K,4K)   \\
p    & (0,p_1) & (0,p_1) & (0,p_1) & (0,p_1)  \\
k & (0,1)& (0,1)& (0,1)& (0,1)\\
\hline
\lam & N_2^+ & N_2^- & N_2^-  & N_2^+  \\
\tau & (0,K) & (-K,0) & (0,K) & (-K,0)   \\
p    & (0,K) & (0,K) & (0,K) & (0,K)  \\
k & (0,1)& (0,1)& (0,1)& (0,1)\\
\hline
\lam & N_3^{+} & N_3^{-} & N_3^{-}  & N_3^{+}  \\
\tau & (0, +\infty) & (-\infty, 0) & (0, +\infty) & (-\infty, 0)   \\
p    & (0,+\infty) & (0, +\infty) & (0,+\infty) & (0,+\infty)  \\
k & 1& 1& 1& 1\\
\hline
 \end{array}
$$ 
\caption{Definition of domains $L_i$}\label{tab:Li}
\end{table}

Table~\ref{tab:Li} should be read by columns. For example, the first column means that
\begin{align}
&L_1 = (L_1 \cap N_1) \sqcup (L_1 \cap N_2^+) \sqcup (L_1 \cap N_3^+), \label{L1decomp} \\
&L_1 \cap N_1 = \{(\tau,p,k) \in N_1 \mid   \tau \in (0, K(k)), \ p \in (0, p_1(k)), \ k \in (0,1)\}, \label{L1N1}\\
&L_1 \cap N_2^+ = \{(\tau,p,k) \in N_2^+ \mid  \tau \in (0, K), \ p \in (0, K(k)), \ k \in (0,1)\}, \label{L1N2}\\
&L_1 \cap N_3^+ = \{(\tau,p, k) \in N_3^+ \mid  \tau \in (0, +\infty), \ p \in (0, +\infty), \ k = 1\}. \label{L1N3}
\end{align}

Decomposition~\eq{tNLi} is schematically shown at Fig.~\ref{fig:decompN}. At this figure the horizontal plane is the state space of pendulum~\eq{pend}, the vertical separating planes are defined by equations $\ts = 0$, $\tc = 0$, the vertical axis is $p$, and the upper surface is defined by the equation $\tt(\lam) = 1$.

\onefiglabel{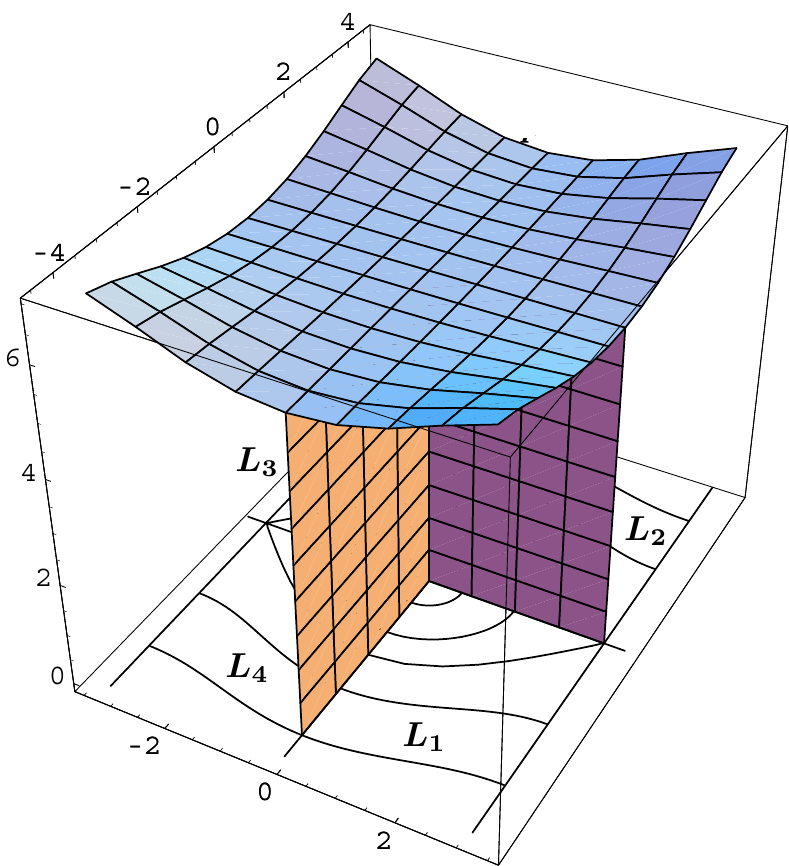}{Decomposition in $N$}{fig:decompN} 

\subsection{Auxiliary lemmas}

\begin{lemma}\label{lem:f1><0}
Let $k \in (0,1)$, and let $p = p_1^1(k)$ be the root of equation~\eq{f1=0}. Then 
\begin{align*}
&p \in (0, p_1^1) \then f_1(p) > 0, \\
&p \in (p_1^1, 3 K) \then f_1(p) < 0.
\end{align*}
\end{lemma}
\begin{proof}
The function $\ds g_1(p) = \frac{f_1(p)}{\cc}$ is increasing at each interval $(K + 2 Kn, 3 K + 2 Kn)$, $n \in \Z$, 
since $\ds\pder{g_1}{p} = \frac{\ssp \ddp}{\ccp} \geq 0$. We have $\ds g_1(p) = \frac{p^3}{3} + o(p^3)$ as $p \to 0$, so  $g_1(p) > 0$ for $p \in (0, K)$, thus $f_1(p) > 0$ for $p \in (0, p_1^1)$. Further, the function $g_1(p)$ changes sign at $p_1^1$, thus $f_1(p)$ changes its sign at $p_1^1$ as well. 
\end{proof}

\begin{lemma}\label{lem:p11}
The function $p = p_1^1(k)$, $k \in (0,1)$, defined by~\eq{f1=0} satisfies the following properties:
\begin{itemize}
\item[$(1)$] 
$p_1^1(k)$  is continuous on the interval $(0, 1)$,
\item[$(2)$] 
$p_1^1(k)$  is smooth on the intervals $(0, k_0) \cup (k_0, 1)$.
\end{itemize}
\end{lemma}
\begin{proof}
Follows by the implicit function theorem.
\end{proof}

A plot of the function $p_1^1(k)$ is given in Fig.~\ref{fig:p11}. Notice the vertical tangent at the point $(k, p) = (k_0, 2 K(k_0))$, $k_0 \approx 0.902$, and the vertical asymptote $k = 1$.

\twofiglabel{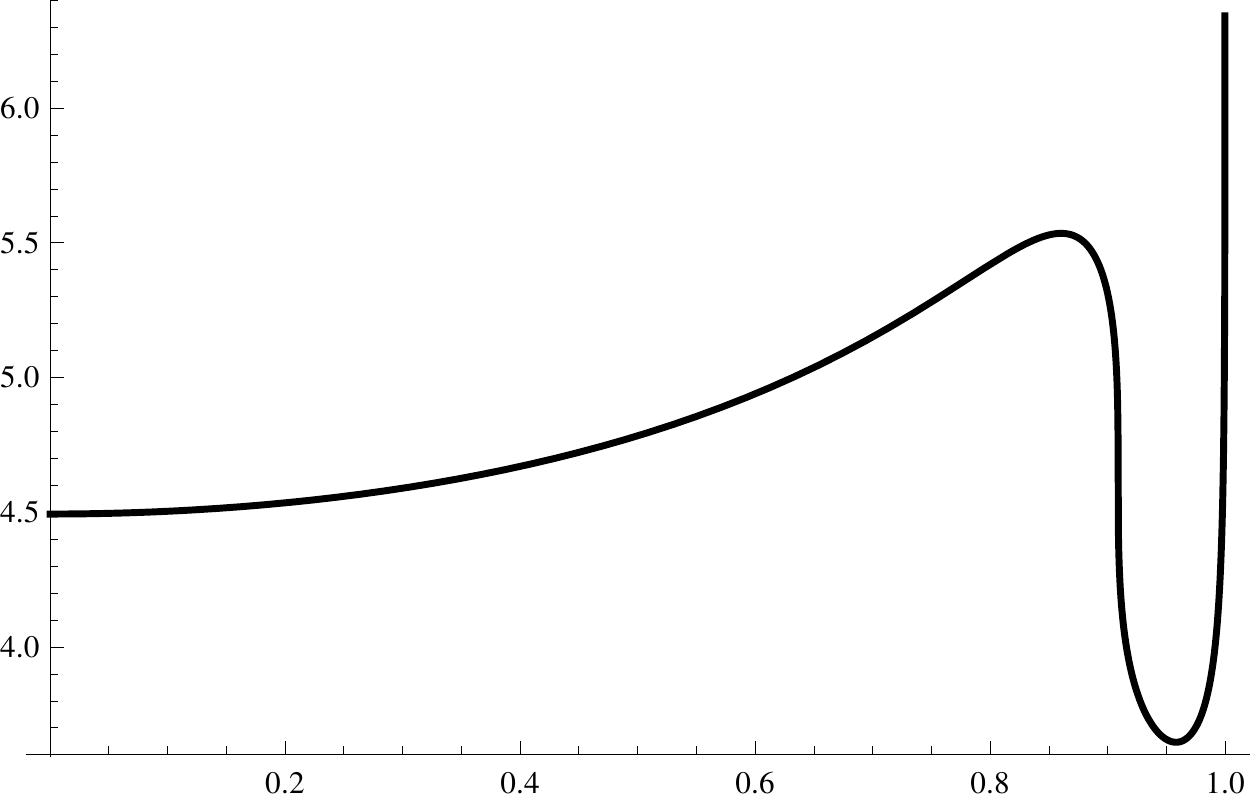}{Plot of the function $p_1^1(k)$}{fig:p11}
{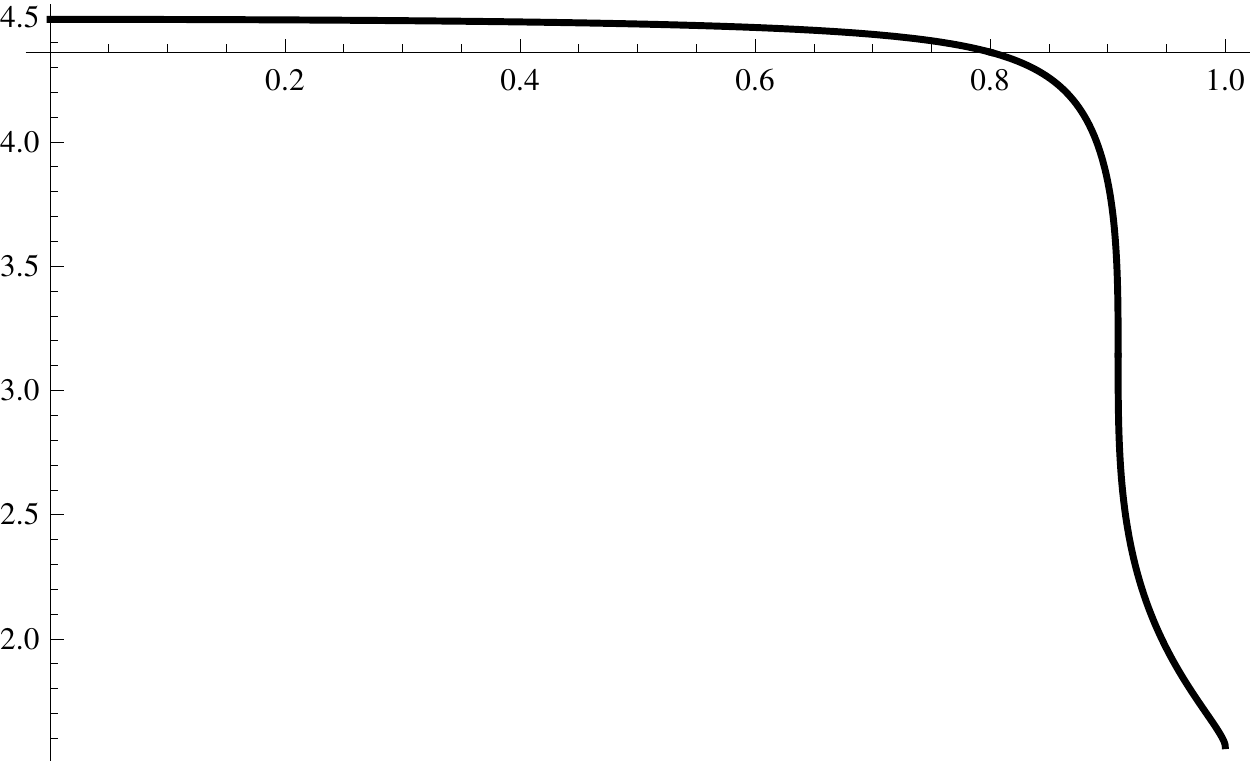}{Plot of the function $\am(p_1^1(k), k)$}{fig:u11k} 

\begin{corollary}\label{cor:p1}
The function $\map{p_1}{(0,1)}{(0, + \infty)}$ given by~\eq{p1(k)N1} is continuous.
\end{corollary}
\begin{proof}
For $k \in [k_0, 1) $, the function $p_1(k) = p_1^1(k)$ is continuous 
by Lemma~\ref{lem:p11}. And for $k \in (0, k_0]$, the function $p_1(k) = 2 K(k)$ is continuous as well. 
\end{proof}

\begin{lemma}\label{lem:ampi2}
Consider sequences  $k^n \in (0, 1)$,  $k^n \to 1 - 0$ and $p^n \in (0, K(k^n))$, $p^n \to + \infty$. Then $\am(p^n, k^n) \to \pi/2$  as $n \to \infty$.
\end{lemma}
Here and below $\am(p,k)$ is Jacobi's amplitude~\cite{whit_watson}.
\begin{proof}
On any converging subsequence of the sequence $u^n = \am(p^n, k^n) \in (0, \pi/2)$ we have $u^{n_m} \to \bu \in [0, \pi/2]$. If $\bu < \pi/2$, then $p^{n_m} = F(u^{n_m}, k^{n_m}) \to F(\bu, 1) < + \infty$, a contradiction.
\end{proof}

Define the function $u_1^1(k) = \am(p_1^1(k),k)$.
By definition~\eq{f1=0} and Propos.~11.6~\cite{el_max},
the function  $u = u_1^1(k)$  is the first positive root of the equation
$$
f_u(u,k) = \sin u \sqrt{1 - k^2 \sin^2 u} - \cos u (2 E(u,k) - F(u,k)),
$$
moreover,
\begin{align*}
&k \in (0, k_0) \then u_1^1 \in (3\pi/2, \pi), \\
&k = k_0 \then u_1^1 = \pi, \\
&k \in (k_0, 1) \then u_1^1 \in (\pi, \pi/2).
\end{align*}

\begin{lemma}\label{lem:u11}
The function $u_1^1(k)$ satisfies the following properties:
\begin{itemize}
\item[$(1)$]
$u_1^1(k)$ decreases as $k \in (k_0, 1)$,
\item[$(2)$]
$\lim_{k \to 1 - 0} u_1^1(k) = \pi/2$.
\end{itemize}
\end{lemma}
\begin{proof}
(1) 
For $k \in (k_0, 1)$, $u \in (\pi/2,\pi)$ we have
\begin{align*}
&\restr{\pder{f_u}{u}}{f_u= 0} = \sqrt{1 - k^2 \sin^2 u} \ \ \frac{\sin^2 u}{\cos u} < 0, \\
&\restr{\pder{f_u}{k}}{f_u= 0} = - \frac{ \sqrt{1 - k^2 \sin^2 u} \ \sin u - \cos u \ F(u,k)}{2 k (1-k^2)} < 0,
\end{align*}
thus 
$$
\der{u_1^1}{k} = - \frac{\partial f_u / \partial k }{\partial f_u / \partial u } < 0.
$$

(2)
The function $u_1^1(k)$ is decreasing for $k \in (k_0,1)$, thus there exists a limit $\lim_{k \to 1 - 0} u_1^1(k) = \bu \in [\pi/2, \pi)$. Assume by contradiction that $\bu > \pi/2$. Then for any $\eps>0$ the domain
$\{(u,k) \in \R^2 \mid u \in (\bu, \bu + \eps), \ k \in (1 - \eps, 1)\}$ contains points such that $f_u(u,k) = 0$. On the other hand, we have:
\begin{align*}
&\lim_{(u,k) \to (\bu, 1-0)} F(u,k) \geq \lim_{k \to 1-0} F(\pi/2,k) = + \infty,\\
&\lim_{(u,k) \to (\bu, 1-0)} f_u(u,k) = - \infty,
\end{align*}
a contradiction.
\end{proof}

 A plot of the function $u_1^1(k)$ is given in Fig.~\ref{fig:u11k}. Notice the vertical tangents at the points $(u,k) = (\pi,k_0)$ and $(u,k) = (\pi/2,1)$.
 
Denote by $\oR$ the completed real line $[-\infty, + \infty] = \{-\infty\} \cup \R \cup \{+\infty\}$, with the  basis of topology consisting of intervals $(a, b)$ and completed rays $[-\infty, b)$, $(a, + \infty]$ for $a, b \in \R$. In the following lemmas we consider a continuous function from a topological space  to the topological space $\oR$.   

\begin{lemma}
\label{lem:tt}
The function $\map{\tt}{N_1 \cup N_2 \cup N_3}{\oR}$ given by~\eq{tt}--\eq{ttN3} is continuous. 
\end{lemma}
\begin{proof}
Notice first that the set
$$
N_1 \cup N_2 \cup N_3 = \{ \lam \in N \mid r > 0, \ (\b,c) \neq (0,0), \, (\pi,0)\}
$$
is open.

If $\lam \in N_1$, then the function $\tt(\lam) = 2 p_1(k)/\sqrt{r}$ is continuous by Cor.~\ref{cor:p1}.

If $\lam \in N_2$, then the function $\tt(\lam) = 2 k K(k)/\sqrt{r}$ is continuous as well.

Let $\lam = (\f, k, r) \in N_3$, $k = 1$, and let $\lam_n \to \lam$ as $n \to \infty$. We show that $\tt(\lam_n) \to \tt(\lam) = + \infty$. 

1) Let $\lam_n = (\f_n, k_n, r_n) \in N_1$ for all $n \in \N$. Then $k_n \to 1$, thus $K(k_n) \to + \infty$, so $(K(k_n), 2 K(k_n)) \ni p_1(k_n) \to + \infty$; moreover, $r_n \to r$ as $n \to \infty$. Consequently, $\tt(\lam_n) = 2 p_1(k_n)/\sqrt{r_n} \to + \infty$.

2) Let $\lam_n = (\f_n, k_n, r_n) \in N_2$ for all $n \in \N$,  then similarly  $\tt(\lam_n) = 2 k_n K(k_n)/\sqrt{r_n} \to + \infty$.

3) Let $\lam_n \in N_3$ for all $n \in \N$,  then    $\tt(\lam_n) = + \infty \to + \infty$.

Thus for any sequence $\lam_n \in N_1 \cup N_2 \cup N_3$ with $\lam_n \to \lam \in N_3$ we have $\tt(\lam_n) \to + \infty$, so the function $\tt$ is continuous on $N_3$.
\end{proof}

Define the following subset in the preimage of the exponential mapping:
\begin{align*}
&K_1 = \{ \lam = (\b, c, r)  \in N \mid \b \in (0, \pi), \ c > 0, \ r > 0, \ \tt(\lam) > 1 \}.
\end{align*}

\begin{lemma}\label{lem:K1}
The set $K_1$ is open.
\end{lemma}
\begin{proof}
The set $K_1 \subset N_1 \cup N_2 \cup N_3$ is determined by a system of strict inequalities for continuous functions, thus it is open (the function $\tt(\lam)$ is continuous by Lemma~\ref{lem:tt}).
\end{proof}

\subsection{Properties of decomposition in preimage \\of exponential mapping}
In this subsection we prove some topological properties of decomposition~\eq{tNLi}. 

\begin{lemma}\label{lem:L1}
The set $L_1$ is open.
\end{lemma}
\begin{proof}
Consider the vector field $\vP = c \pder{}{\b} - r \sin \b \pder{}{c} \in \Vec(N)$ determined by the equation of pendulum~\eq{pend}. We show that 
\be{L1K1}
L_1 = e^{-1/2 \vP} (K_1),
\ee
where $\map{e^{-1/2 \vP}}{N}{N}$ is the flow of the vector field $\vP$ for the time $-1/2$.

Since energy~\eq{E_r} is an integral of pendulum, then  $\vP E  = 0$, thus $e^{t \vP}(N_i) = N_i$, $i = 1, 2, 3$. Further, the coordinates $(\f, p, k)$ rectify the flow of the vector field $\vP$ (see~\eq{rectify}), so in these coordinates $\vP = \pder{}{\f}$.
Since
\begin{align*}
&L_1 \cap N_1 = \{ \lam \in N_1 \mid \f \in (-1/2, K/\sqrt r - 1/2), \ p \in (0, p_1(k)), \ k \in (0,1)\}, \\
&K_1 \cap N_1 = \{ \lam \in N_1 \mid \f \in (0, K/\sqrt r), \ p \in (0, p_1(k)), \ k \in (0,1)\}, 
\end{align*}
it is obvious that $L_1 \cap N_1 = e^{-1/2 \vP} (K_1 \cap N_1)$.

Similarly it follows that $L_1 \cap N_i = e^{-1/2 \vP} (K_1 \cap N_i)$
 for $i = 2, 3$.
  
Then equality~\eq{L1K1} follows. Since the set $K_1$ is open and the flow $\map{e^{-1/2 \vP}}{N}{N}$ is a diffeomorphism, then the set $L_1$ is open as well.
\end{proof}

\begin{lemma}\label{lem:L1con}
The set $L_1$ is arcwise connected.
\end{lemma}
\begin{proof}
It is obvious from equalities~\eq{L1N1}--\eq{L1N3} that the sets $L_1 \cap N_i$, $i = 1, 2, 3$, are arcwise connected. Since any point in $L_1 \cap N_3$ can be connected with some close points in $L_1 \cap N_1$ and $L_1 \cap N_2$ by a continuous curve, then the set $L_1$ is arcwise connected.
\end{proof}

In Sec.~9~\cite{el_max} were defined discrete symmetries of the elastic problem --- reflections $\eps^1$, $\eps^2$, $\eps^3$ that act both in preimage and image of the exponential mapping, and commute with it.

\begin{lemma}\label{lem:epsinsym}
\begin{itemize}
\item[$(1)$]
The mappings $\map{\eps^i}{N}{N}$, $i = 1, 2, 3$, are diffeomorphisms.
\item[$(2)$]
The reflections $\eps^i$ permute the sets $L_j$ as shown by Table~\ref{tab:epsiLj}.
\end{itemize}
\end{lemma}

\begin{table}[htbp]
$$
\begin{array}{|c|c|c|c|c|}
\hline
L_j & L_1 & L_2 & L_3 & L_4  \\
\hline
\eps^1(L_j) & L_2 & L_1 & L_4   & L_3   \\
\hline
\eps^2(L_j) & L_4 & L_3 & L_2 & L_1   \\
\hline
\eps^3(L_j)   & L_3 & L_4 & L_1 & L_2  \\
\hline
 \end{array}
$$ 
\caption{Action of $\eps^i$ on $L_j$}\label{tab:epsiLj}
\end{table}

\begin{proof}
(1) By the definition given in Subsec.~9.7~\cite{el_max}, we have $\mapto{\eps^1}{(\b, c, r)}{(\b_1, - c_1, r)}$, where $e^{\vP}(\b, c, r) = (\b_1, c_1, r)$. Since $e^{\vP}$ is smooth, then $\eps^1$ is smooth as well. Moreover, we have $\eps^1 \circ \eps^1 = \Id$, thus $\eps^1$ is a diffeomorphism. Similarly, $\eps^2$ and $\eps^3$ are diffeomorphisms.

(2) The reflection $\eps^1$ preserves the coordinates $k$, $p$ and acts as follows on the coordinate $\tau$ of a point $\lam = (p, \tau, k) \in N_1\cup N_2 \cup N_3$:
\begin{align*}
&\lam \in N_1 \then \mapto{\eps^1}{\tau}{2 K - \tau}, \\
&\lam \in N_2 \cup N_3 \then \mapto{\eps^1}{\tau}{- \tau}.
\end{align*}
Thus $\eps^1(L_1 \cap N_i) = L_2 \cap N_i$, $i = 1, 2, 3$. So $\eps(L_1) = L_2$.

Similarly one proves the remaining entries of Table~\ref{tab:epsiLj}.
\end{proof}

\begin{proposition}\label{propos:Li}
The sets $L_i$, $i = 1, \dots, 4$, are open and arcwise connected.
\end{proposition}
\begin{proof}
Follows from Lemmas~\ref{lem:L1}--\ref{lem:epsinsym}.
\end{proof}

\subsection{Decomposition in image of exponential mapping}

Recall that the time 1 attainable set of system~\eq{sys1}--\eq{sys3} is
$$
\CA =
\{ (x,y,\t) \in M 
\mid
x^2 + y^2 < 1
\text{ or }
(x,y,\t) = (1, 0, 0) \}.
$$
Consider the following decomposition of this set:
\begin{align}
&\CA = \tM \sqcup M', \label{AMtilde}\\
&\tM = \{ q \in \CA   \mid  P(q)  \sin (\t/2) \neq 0\},  \nonumber\\
&M' = \{ q \in \CA   \mid P(q) \sin(\t/2) = 0\},  \nonumber\\
&M_{\pm} = \{ q \in M \mid \t \in (0, 2 \pi), \ x^2 + y^2 < 1, \ \sgn P(q) = \pm 1 \}, \nonumber\\
&\tM = M_+ \sqcup M_-. \label{tM+-}
\end{align}
The function $P(q) = x \sin(\t/2) - y \cos(\t/2)$ was introduced in~\cite{el_max}, it is defined on $M$ up to sign. If $\t \in (0, 2 \pi)$ as in $M_{\pm}$, then the function $P(q)$ is well-defined.

Decomposition~\eq{tM+-} is   shown in Fig.~\ref{fig:decompM}.

\onefiglabel{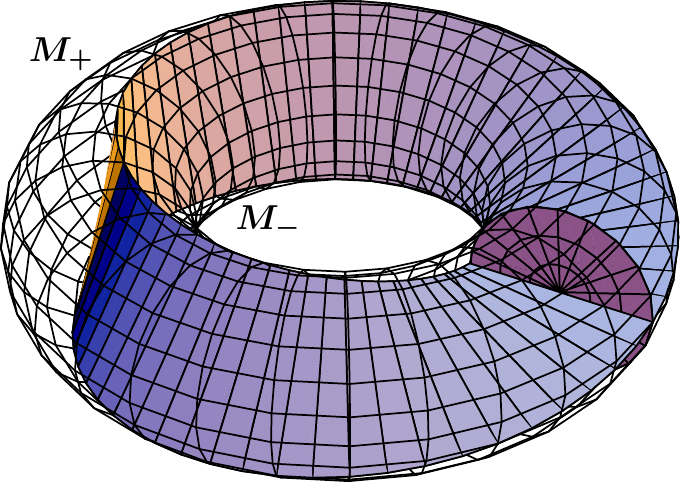}{Decomposition in $M$}{fig:decompM} 

\begin{lemma}\label{lem:M+-}
The sets $M_+$ and $M_-$ are open, arcwise connected, and simply connected.
\end{lemma}
\begin{proof}
Obvious.
\end{proof}

\begin{lemma}\label{lem:epsiM+-}
The reflections $\eps^i$ permute the sets $M_{\pm}$ as shown by Table~\ref{tab:epsiM+-}.
\end{lemma}

\begin{table}[htbp]
$$
\begin{array}{|c|c|c|}
\hline
M_{\pm}& M_+ & M_-  \\
\hline
\eps^1(M_{\pm}) & M_-  & M_+    \\
\hline
\eps^2(M_{\pm}) & M_-  & M_+   \\
\hline
\eps^3(M_{\pm})   & M_+  & M_-  \\
\hline
 \end{array}
$$ 
\caption{Action of $\eps^i$ on $M_{\pm}$}\label{tab:epsiM+-}
\end{table}

\begin{proof}
Action of reflections $\eps^i$ in $M$ is given by
formulas (9.10)--(9.12)~\cite{el_max}, with appropriate choice of  the branch of $\t \in (0, 2 \pi)$:
\begin{align}
&\mapto{\eps^1}{\vect{\t \\ x \\ y}}{\vect{2 \pi-\t \\ x \cos \t + y \sin \t \\ -x \sin \t + y \cos \t}}, \label{eps1txy} \\
&\mapto{\eps^2}{\vect{\t \\ x \\ y}}{\vect{\t \\ x \cos \t + y \sin \t \\ x \sin \t - y \cos \t}}, \label{eps2txy} \\
&\mapto{\eps^3}{\vect{\t \\ x \\ y}}{\vect{2 \pi-\t \\ x  \\ - y}}. \label{eps3txy}
\end{align}
These formulas show that the reflections $\eps^i$ preserve the restrictions $\t \in (0, 2 \pi)$ and $x^2 + y^2 < 1$, and
imply the following transformation rules for the function~$P$:
$$
\mapto{\eps^1}{P}{-P}, \qquad \mapto{\eps^2}{P}{-P}, \qquad \mapto{\eps^3}{P}{P}. 
$$
Then $\eps^1(M_{\pm}) = \eps^2(M_{\pm}) = M_{\mp}$ and $\eps^3(M_{\pm}) = M_{\pm}$,
which gives Table~\ref{tab:epsiM+-}.
\end{proof}

\begin{lemma}\label{lem:epsiMdiffeo}
The mappings $\map{\eps^i}{M}{M}$, $i = 1, 2, 3$, are diffeomorphisms.
\end{lemma}
\begin{proof}
The reflections $\eps^i$ are smooth by formulas (9.10)--(9.12)~\cite{el_max} and satisfy $\eps^i \circ \eps^i = \Id$.
\end{proof}

\begin{lemma}\label{lem:ExpLiM+-}
The action of the exponential mapping on the sets $L_i$   is shown by Table~\ref{tab:ExpLiM+-}.
\end{lemma}

\begin{table}[htbp]
$$
\begin{array}{|c|c|c|c|c|}
\hline
L_i & L_1 & L_2  & L_3 & L_4\\
\hline
\Exp(L_i) & M_+  & M_-  & M_+  & M_-   \\
\hline
 \end{array}
$$ 
\caption{Action of $\Exp$ on $L_i$}\label{tab:ExpLiM+-}
\end{table}

\begin{proof}
First we show that $\Exp(L_1) \subset M_+$.

Let $\lam \in N_1$. It follows from the parameterisation of extremal trajectories obtained in~\cite{el_max} that
\begin{align}
&\sin \frac{\t}{2} = \frac{2 k \ss \dd \tc}{\Delta}, \qquad \Delta = 1 - k^2 \ssp \tsp, \label{sinth2N1}\\
&P = \frac{4 k \ts \td f_1(p)}{\Delta}. \label{PN1}
\end{align}

Let $\lam \in L_1 \cap N_1$. Then $\tau \in (0, K)$, thus $\ts > 0$, $\td > 0$, $\tc > 0$.
Moreover, since $p \in (0, p_1(k))$, then $f_1(p) > 0$ (see Lemma~\ref{lem:f1><0}) and $\ss > 0$, $\dd > 0$. Thus $\sin \frac{\t}{2} > 0$, $P > 0$, so $\Exp(L_1 \cap N_1) \subset M_+$.

Let $\lam \in N_2$. Then
\begin{align}
&\sin \frac{\t}{2} = \frac{2 \cc \ss \td}{\Delta},  \label{sinth2N2}\\
&P = \frac{4 k \ts \tc f_2(p)}{\sqrt r \Delta}, \label{PN2}\\
&f_2(p) = (k^2 \ss \cc + \dd ((2-k^2)p - 2 \E(p)))/k. \nonumber
\end{align}
Similarly to the case $\lam \in L_1 \cap N_1$, these formulas imply that $\Exp(L_1 \cap N_2) \subset M_+$.

If $\lam \in N_3$, then formulas~\eq{sinth2N2}, \eq{PN2} remain valid with $k = 1$, and similarly to the case $\lam \in L_1 \cap N_2$ it follows that $\Exp(L_1 \cap N_3) \subset M_+$.

Thus $\Exp(L_1) \subset M_+$.

We have $\eps^i \circ \Exp = \Exp \circ \eps^i$ on $N$ (see Propos. 9.2~\cite{el_max}). Then by virtue of Lemmas~\ref{lem:epsinsym} and~\ref{lem:epsiM+-} we get
$$
\Exp(L_2) = \Exp \circ \eps^1(L_1) = \eps^1 \circ \Exp(L_1) \subset \eps^1(M_+) = M_-.
$$
Similarly it follows that $\Exp(L_3) \subset M_+$, $\Exp(L_4) \subset M_-$.
\end{proof}

\section{Diffeomorphic properties \\of exponential mapping}\label{sec:diffeo}
In this section we prove the main result of this work.

\begin{theorem}\label{th:Expdiffeo}
The following mappings are diffeomorphisms:
$$
\map{\Exp}{L_1}{M_+}, \quad \map{\Exp}{L_2}{M_+}, \quad 
\map{\Exp}{L_3}{M_+}, \quad \map{\Exp}{L_4}{M_-}.
$$
\end{theorem}

By virtue of Lemmas~\ref{lem:epsinsym}, \ref{lem:epsiMdiffeo}, \ref{lem:epsiM+-}, it is enough to prove the following statement.

\begin{proposition}\label{propos:ExpdiffeoL1M+}
The mapping $\map{\Exp}{L_1}{M_+}$ is a  diffeomorphism.
\end{proposition}

We prove this statement by applying the following Hadamard's global inverse function theorem.

\begin{theorem}[Th. 6.2.8~\cite{inverse}]\label{th:global_diffeo}
Let $X$, $Y$ be smooth manifolds and let $\map{F}{X}{Y}$ be a smooth mapping such that:
\begin{enumerate}
\item
$\dim X = \dim Y$, 
\item
$X$ and  $Y$ are arcwise connected,
\item
$Y$ is simply connected,
\item
$F$ is nondegenerate,
\item
$F$ is proper (i.e., preimage of a compact is a compact).
\end{enumerate}
Then $F$ is a diffeomorphism.
\end{theorem}

Now we check hypotheses 4 and 5 of Th.~\ref{th:global_diffeo} for the mapping $\map{\Exp}{L_1}{M_+}$.

\begin{proposition}\label{propos:ExpL1M+nondeg}
The mapping $\map{\Exp}{L_1}{M_+}$ is nondegenerate.
\end{proposition}
\begin{proof}
Theorem~5.1~\cite{el_conj} gives the following lower bound on the first conjugate time $\tconj(\lam)$ along extremal trajectory $\Exp_t(\lam)$:
$$
\forall \lam \in N \qquad \tconj(\lam) \geq \tt(\lam).
$$
Let $\lam \in L_1$, then $\tt(\lam) > 1$, thus $\tconj(\lam) > 1$. This means that the differential $\map{\Exp_{*\lam}}{T_{\lam}N}{T_qM}$, $q = \Exp(\lam)$, is nondegenerate.
\end{proof}

\begin{proposition}\label{propos:ExpL1M+proper}
The mapping $\map{\Exp}{L_1}{M_+}$ is proper.
\end{proposition}
\begin{proof}
Let $K \subset M_+$ be a compact. Denote the function $R(q) = x^2 + y^2 - 1$. There exists $\eps> 0$ such that for any $q \in K$
\be{sinthPR}
\sin \frac{\t}{2} \geq \eps, \qquad P(q) \geq \eps, \qquad - \eps \geq R(q) \geq -1.
\ee
We prove that the preimage $S = \Exp^{-1}(K) \subset L_1$ is compact, i.e., bounded and closed.

By contradiction, suppose first that $S$ is unbounded, then it contains a sequence $\lam^n = (\tau^n, p^n, k^n) \to \infty$.

If $k^n$ is separated from 1, then the sequences $\tau^n \in (0, K(k^n))$ and $p^n \in (0, 2K(k^n))$ are bounded, thus $\lam^n$ is bounded, a contradiction. Thus $k^n \to 1$ on a subsequence of $\lam^n$ (we keep the notation $\lam^n$ for this subsequence). Then $(\tau^n, p^n) \to \infty$.

1) Let $\lam^n \in N_1$ for all $n \in \N$. Then we have decompositions~\eq{sinth2N1}, \eq{PN1} and obtain from parameterisation of extremal trajectories~\cite{el_max}
\begin{align}
&R = \frac{16 \E(p) (\E(p) - p)}{r}  
+ \frac{16 k^2 \ss \dd f_1(p) \tsp}{r \D}.
 \label{RN1}
\end{align} 

1.1) Let $\tau^n \to \btau \in [0, + \infty)$, $p^n \to \infty$, $k^n \to 1$.
By Lemmas~\ref{lem:ampi2}, \ref{lem:u11}, we have $\am(p^n, k^n) \to \pi/2$, thus $\sn(p^n) \to 0$, $\sn(\tau^n) \to \sn(\btau, 1) = \tanh \btau < 1$, $\D \to 1 - \tanh^2 \bt > 0$, $\dn(p^n,k^n) \to 1$. By virtue of~\eq{sinth2N1}, we have $\sin(\t^n/2) \to 0$, which contradicts~\eq{sinthPR}.

1.2) The case  $\tau^n \to + \infty$, 
$p^n \to \bp \in [0, \infty)$
is considered similarly to the case 1.1).

1.3)
Let 
$\tau^n \to + \infty$, 
$p^n \to  +\infty$, $k^n \to 1$. 
Then $\am(\tau^n, k^n) \to \pi/2$, 
$\am(p^n, k^n) \to \pi/2$ by Lemmas~\ref{lem:ampi2}, \ref{lem:u11}.

We have $\E(p) = \E(p^n, k^n) = E(\am(p^n, k^n),k^n) \to E(\pi/2, 1) = 1$, thus\\
$
\ds\frac{\E(p)(p - \E(p))}{4 p^2} \to 0
$.
By virtue of the inequalities~\eq{sinthPR} for $R$, there exists a subsequence $\lam^n$ on which $R(q^n) \to - \eps_1 \leq - \eps$. Then the second term in~\eq{RN1} tends to $-\eps_1< 0$, which is impossible since this term is positive.

So the set $S \cap N_1$ does not contain sequences $\lam^n \to \infty$, i.e., it is bounded.

2) Similarly it follows that the sets $S \cap N_2^+$ and $S \cap N_3^+$ are bounded.
Thus the set $S \cap N_2^+$ is bounded.

3) The sets  $S \cap N_1$, $S \cap N_2^+$, $S \cap N_3^+$ are bounded, thus $S$ is bounded as well.

Now we prove that $S$ is closed. Let $\lam^n \in S$, $\lam^n = (p^n, \tau^n, k^n) \to (\bp, \btau, \bk) = \blam \in \cl(L_1)$. We show that $\blam \in S$. 

1) Let $\lam^n \in S \cap N_1$. 

1.1) If $\blam \in L_1$, then $\bq = \Exp(\blam) \in M_+$, on a subsequence $\Exp(\lam^n) \to \bq$, thus $\bq \in K$ and $\blam \in S$.

1.2) Let 
$\blam \notin L_1$,
thus
\be{kbar=0}
\bk = 0 \quad \vee \quad \bk = 1 \quad \vee \quad \btau = 0 \quad \vee \quad \btau = K \quad \vee \quad \bp = 0 \quad \vee \quad \bp = p_1.     
\ee
Each of these conditions leads to a contradiction with inequalities~\eq{sinthPR}.
For example, let $\bk = 0$. Then $k^n \to 0$, $p^n \to \bp$, $\tau^n \to \btau$, thus $\D \to 1$. By~\eq{sinth2N1},  $\sin(\t/2) \to 0$, which contradicts~\eq{sinthPR}.
All other cases in~\eq{kbar=0} are considered similarly. 
Thus $\blam \in S$ in the case $\lam^n \in S \cap N_1$.

2) Similarly, the inclusion $\blam \in S$ follows in the cases $\lam^n \in S \cap N_2$ and $\lam^n \in S \cap N_3$.

We proved that the set $S$ is closed. Since it is bounded as well, it is compact. Thus the mapping $\map{\Exp}{L_1}{M_+}$ is proper.
\end{proof}

Now we can prove Proposition~\ref{propos:ExpdiffeoL1M+}.
\begin{proof}
We check hypotheses of Th.~\ref{th:global_diffeo} for the mapping $\map{\Exp}{L_1}{M_+}$. The sets $L_1$ and $M_+$ are open subsets in a 3-dimensional linear space (Lemmas~\ref{lem:L1} and~\ref{lem:M+-}).  Moreover, we have:

\begin{enumerate}
\item
$\dim L_1 = \dim M_+ = 3$,
\item
$L_1$ and $M_+$ are arcwise connected (Lemmas~\ref{lem:L1con} and~\ref{lem:M+-}),
\item
$M_+$ is simply connected (Lemma~\ref{lem:M+-}),
\item
the mapping $\map{\Exp}{L_1}{M_+}$ is nondegenerate (Propos.~\ref{propos:ExpL1M+nondeg}),
\item
the mapping $\map{\Exp}{L_1}{M_+}$ is proper (Propos.~\ref{propos:ExpL1M+proper}).
\end{enumerate}
By Theorem~\ref{th:global_diffeo}, the mapping $\map{\Exp}{L_1}{M_+}$ is a diffeomorphism.
\end{proof}

By virtue of Lemmas~\ref{lem:epsinsym}, \ref{lem:epsiMdiffeo}, \ref{lem:epsiM+-}, Theorem~\ref{th:Expdiffeo} follows. This theorem implies that
\be{ExpNtildeMtilde}
\Exp(\tN) = \tM.
\ee

\section{Action of exponential mapping \\on the boundary of diffeomorphic domains}\label{sec:boundary}

Define the following subsets in the boundary of the set $\tM$:
\begin{align*}
&M_P = \{q \in \CA \mid P(q) = 0 \}, \\
&M_{\t} = \{q \in \CA \mid \sin(\t/2) = 0 \}, \\
&V = \{(x, y, \t) = (1, 0, 0)\}.
\end{align*}

\begin{proposition}\label{propos:ExpN'M'}

We have
\be{ExpN'M'}
\Exp(N') = M'.
\ee
\end{proposition}
\begin{proof}
Recall decomposition~\eq{N'} of the set $N'$.

It follows from definitions~\eq{N4}--\eq{N7} and the parameterisation of extremal trajectories~\cite{el_max} that
\begin{align}
&\Exp(N_4) = \Exp(N_5) = \Exp(N_7) = V,  \label{ExpN457} \\
&\Exp(\hN_6) \subset M_P \subset M'.  \label{ExpN6}
\end{align}

Further, it follows from formulas~\eq{sinth2N1}--\eq{PN2} that
\be{ExpMthMPM'}
\Exp(N_{1-3}') \subset M_{\t} \cup M_P \subset M'.
\ee

Then we obtain from~\eq{ExpN457}--\eq{ExpMthMPM'} that $\Exp(N') \subset M'$. 
But the mapping $\map{\Exp}{\hN}{\CA}$ is surjective, then equalities~\eq{NhatNtilde}, \eq{AMtilde}, \eq{ExpNtildeMtilde} imply equality~\eq{ExpN'M'}.
\end{proof}

\section{Optimal elasticae \\for various boundary conditions}\label{sec:various}
In this section we describe optimal trajectories for various terminal points $q_1 = (x_1, y_1, \t_1) \in \CA$.

\subsection{Generic boundary conditions}\label{subsec:generic}

let $q_1 \in M_+$, then by Th.~\ref{th:Expdiffeo} there exist a unique $\lam_1 \in L_1$ and a unique $\lam_3 \in L_3$ such that $\Exp(\lam_1) = \Exp(\lam_3) = q_1$. Since $\Exp(L_2) = \Exp(L_4) = M_-$ and $\Exp(N') = M'$, the equation
\be{Explamq1}
\Exp(\lam) = q_1, \qquad \lam \in \hN,
\ee
has only two solutions, $\lam_1$ and $\lam_3$. By virtue of existence of optimal trajectory connecting $q_0$ to $q_1$, it should be $q^1(t) = \Exp_t(\lam_1)$ or $q^3(t) = \Exp_t(\lam_3)$. In order to find the optimal trajectory, one should compare the costs 
$J[q^i] = \frac 12 \int_0^1 (c_t^i)^2 \, dt$, $i = 1, 3$, of the competing candidates $q^1(t)$ and $q^3(t)$ and choose the less one, see Fig.~\ref{fig:2elastica}.

\onefiglabel{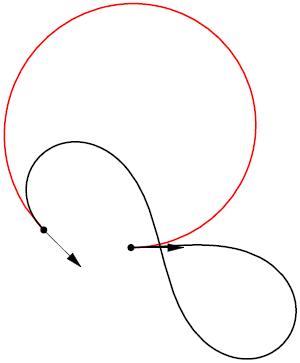}{Competing elasticae with the same boundary conditions}{fig:2elastica}

If $J[q^1] \neq J[q^3]$, then the optimal trajectory is unique. 

If $J[q^1] = J[q^3]$, then there are two optimal trajectories coming to the point $q_1$.
(See example of the corresponding elasticae at Fig.~\ref{fig:2optimal}).  
Such points $q_1$ are  Maxwell points that arise due to some unclear reason different from the reflections $\eps^i$. 

\onefiglabel{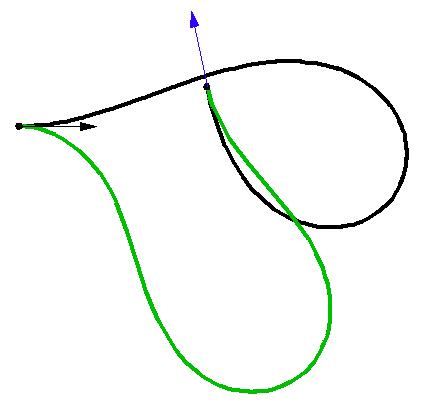}{Two optimal non-symmetric elasticae with the same boundary conditions}{fig:2optimal}

If $q_1 \in M_-$, then the analysis of optimal trajectories is similar to the case $q_1 \in M_+$.

A.~Ardentov designed a software in Mathematica~\cite{math} for numerical computation of optimal elasticae for $q_1 \in \tM$ by solving the equation~\eq{Explamq1}, the software and algorithm are described in~\cite{findelastica}.
An example of a sequence of optimal elasticae computed by this software for a given sequence of boundary conditions is given in Fig.~\ref{fig:family}. 

\onefiglabel{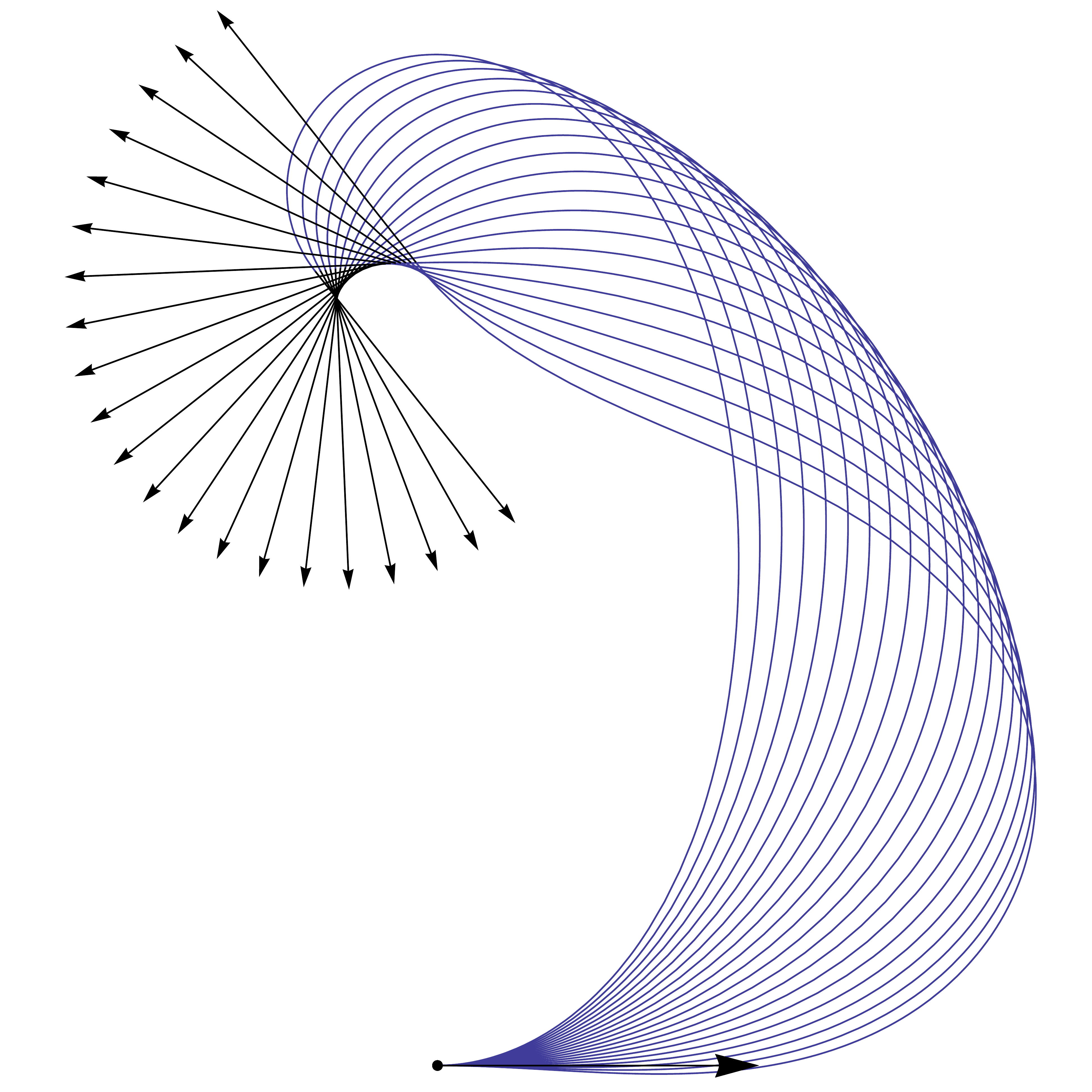}{Sequence of optimal elasticae}{fig:family}

\subsection{The case $y_1 = 0$, $\t_1 = \pi$}
\subsubsection{The case $x_1 > 0$}
We have $P(q_1) = x_1 > 0$, thus $q_1 \in M_+$. As shown in Subsec.~\ref{subsec:generic}, the equation~\eq{Explamq1} has solutions $\lam_1 \in L_1$ and $\lam_3 \in L_3$. By~\eq{eps3txy}, $\eps^3(q_1) = q_1$, thus $\eps^3(\lam_1) = \lam_3$. Then the trajectories $q_1(t) = \Exp_t(\lam_1)$ and $q_3(t) = \Exp_t(\lam_3)$ have the same cost, thus they are both optimal. The corresponding optimal inflectional elasticae are symmetric w.r.t. the line $y=0$, see Fig.~\ref{fig:x1g0y10th1pi}. 

\subsubsection{The case $x_1 < 0$}
This case is similar to the case $x_1 > 0$, see Fig.~\ref{fig:x1l0y10th1pi}.

\subsubsection{The case $x_1 = 0$}
It follows from results of Secs.~11.6--11.10~\cite{el_max} that in the case $(x_1, y_1, \t_1) = (0, 0, \pi)$ the equation~\eq{Explamq1} has solutions $\lam = (p, \tau, k) \in N_1$ with $\ts = 0$, $1 - 2 k^2 \ssp = 0$, $2 \E(p) - p = 0$. Then there exists a unique optimal elastica shown in Fig.~\ref{fig:x1=0y10th1pi}.

\begin{figure}[htbp]
\includegraphics[height=0.3\textwidth]{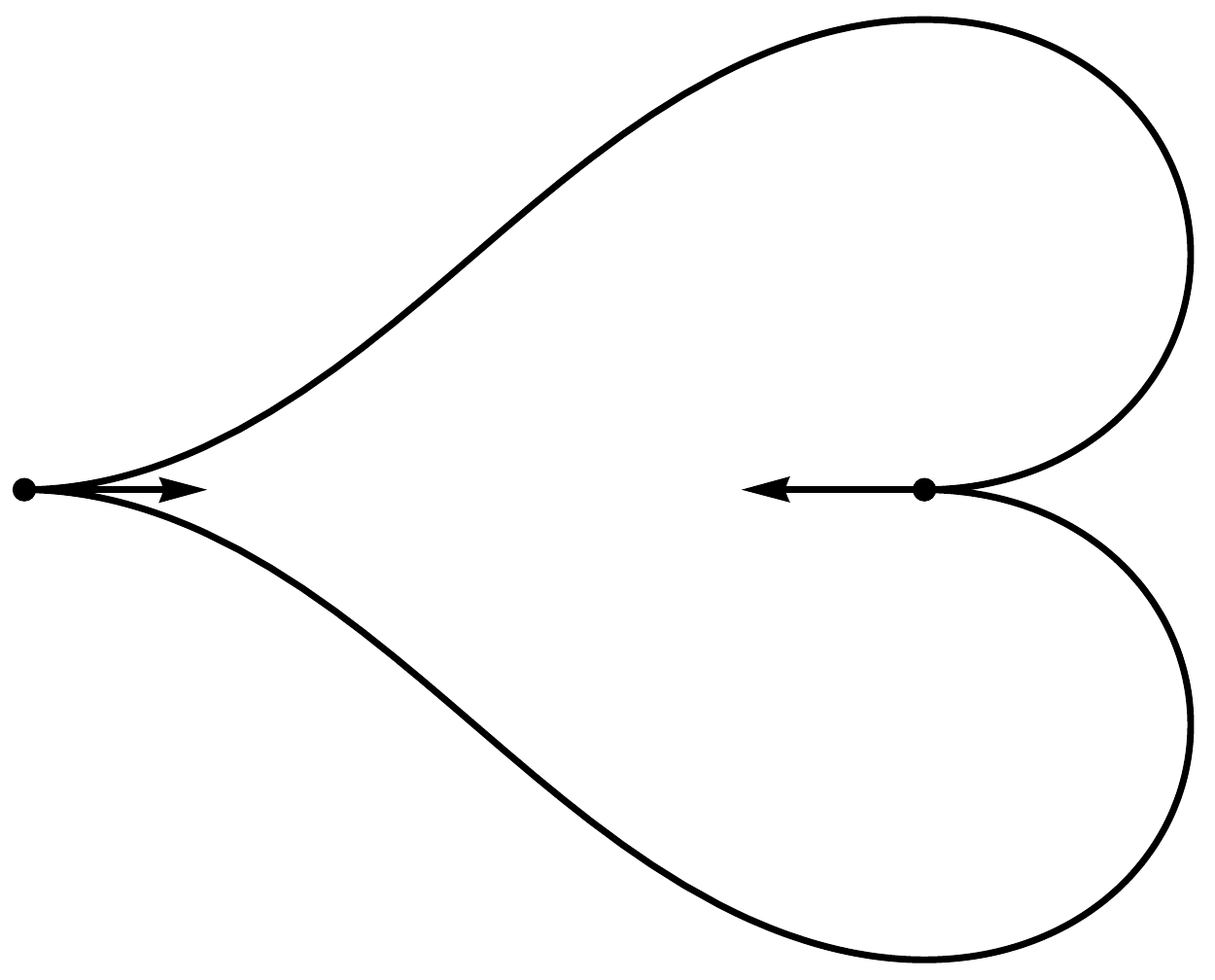}
\hfill
\includegraphics[height=0.25\textwidth]{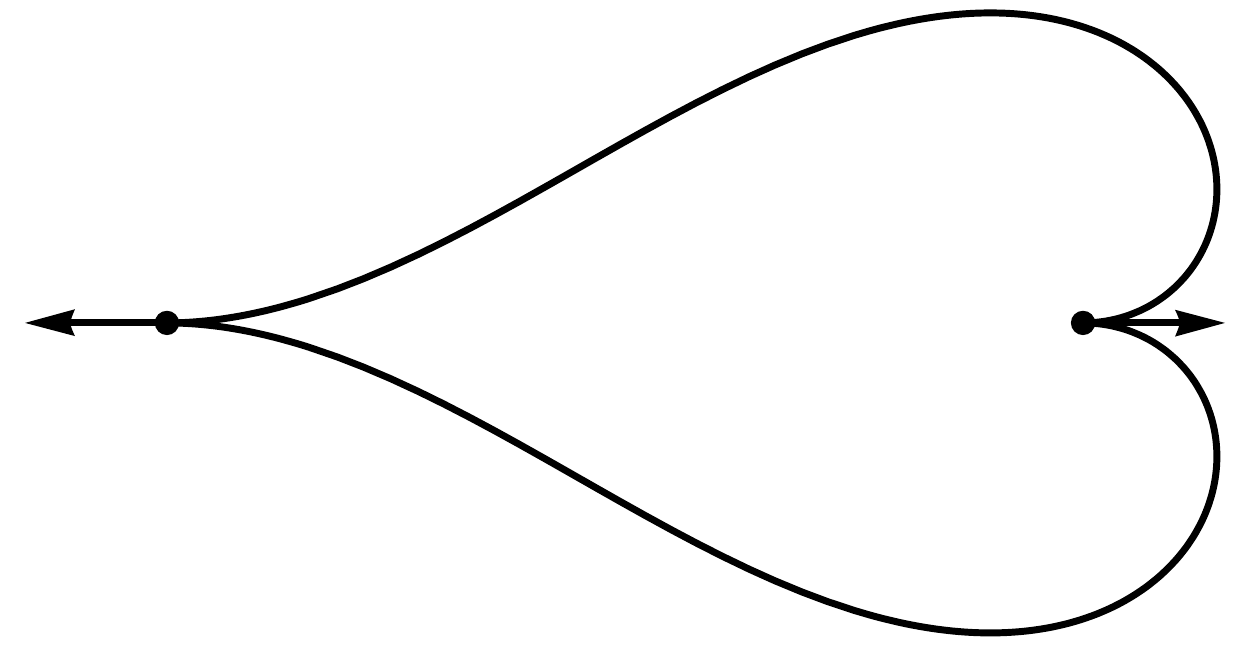}
\\
\parbox[t]{0.45\textwidth}{\caption{Optimal elasticae for $(x_1, y_1, \t_1) = (x_1, 0, \pi)$, $x_1 > 0$}\label{fig:x1g0y10th1pi}}
\hfill
\parbox[t]{0.45\textwidth}{\caption{Optimal elasticae for $(x_1, y_1, \t_1) = (x_1, 0, 0)$, $x_1 < 0$}\label{fig:x1l0y10th1pi}}
\end{figure}

\onefiglabel{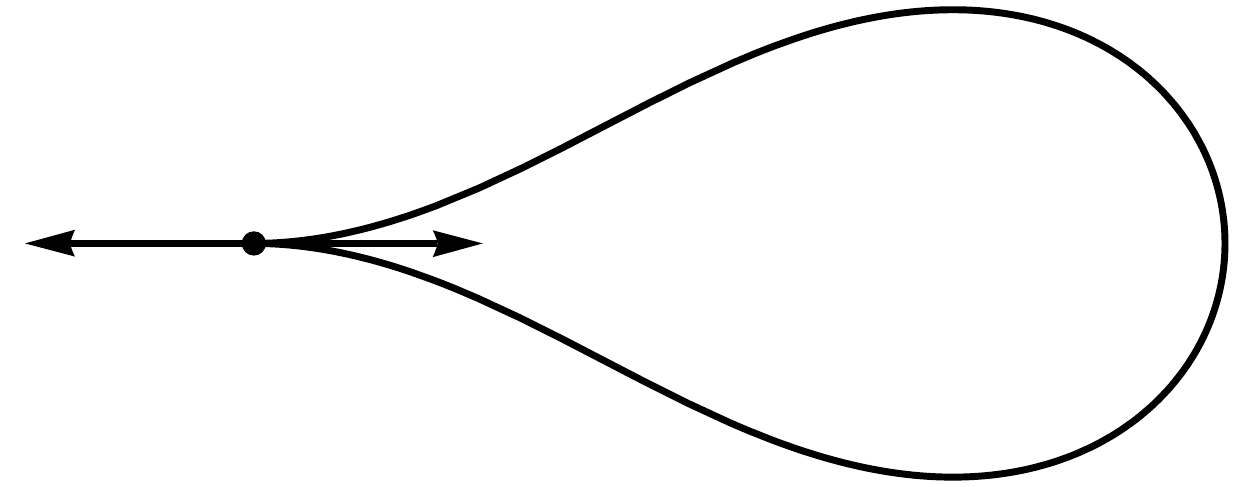}{Optimal elastica for $(x_1, y_1, \t_1) = (0, 0, \pi)$}{fig:x1=0y10th1pi}

\subsection{The case $y_1 = \t_1 = 0$}
\subsubsection{The case $x_1 = 0$}
This case was studied in~\cite{el_closed}, it was shown that there exist two optimal elasticae --- circles symmetric w.r.t. the line $y = 0$.


\subsubsection{The case $x_1 \neq 0$}
One can show that in the case $x_1 > 0$ there are two or four optimal elasticae:
there exists $x_* \in (0.4, 0.5)$ such that
\begin{itemize}
\item
if $x_1 \in (0, x_*)$, then there are two optimal non-inflectional elasticae, 
see  Fig.~\ref{fig:x1g0y10th102opt2},
\item
if $x_1 = x_*$, then there are four optimal   elasticae (two inflectional and two non-inflectional ones), 
see  Fig.~\ref{fig:x1g0y10th102opt4},
\item
if $x_1 \in (x_*, 1)$, then there are two optimal inflectional elasticae, 
see  Fig.~\ref{fig:x1g0y10th102opt1}.
\end{itemize}

\twofiglabel{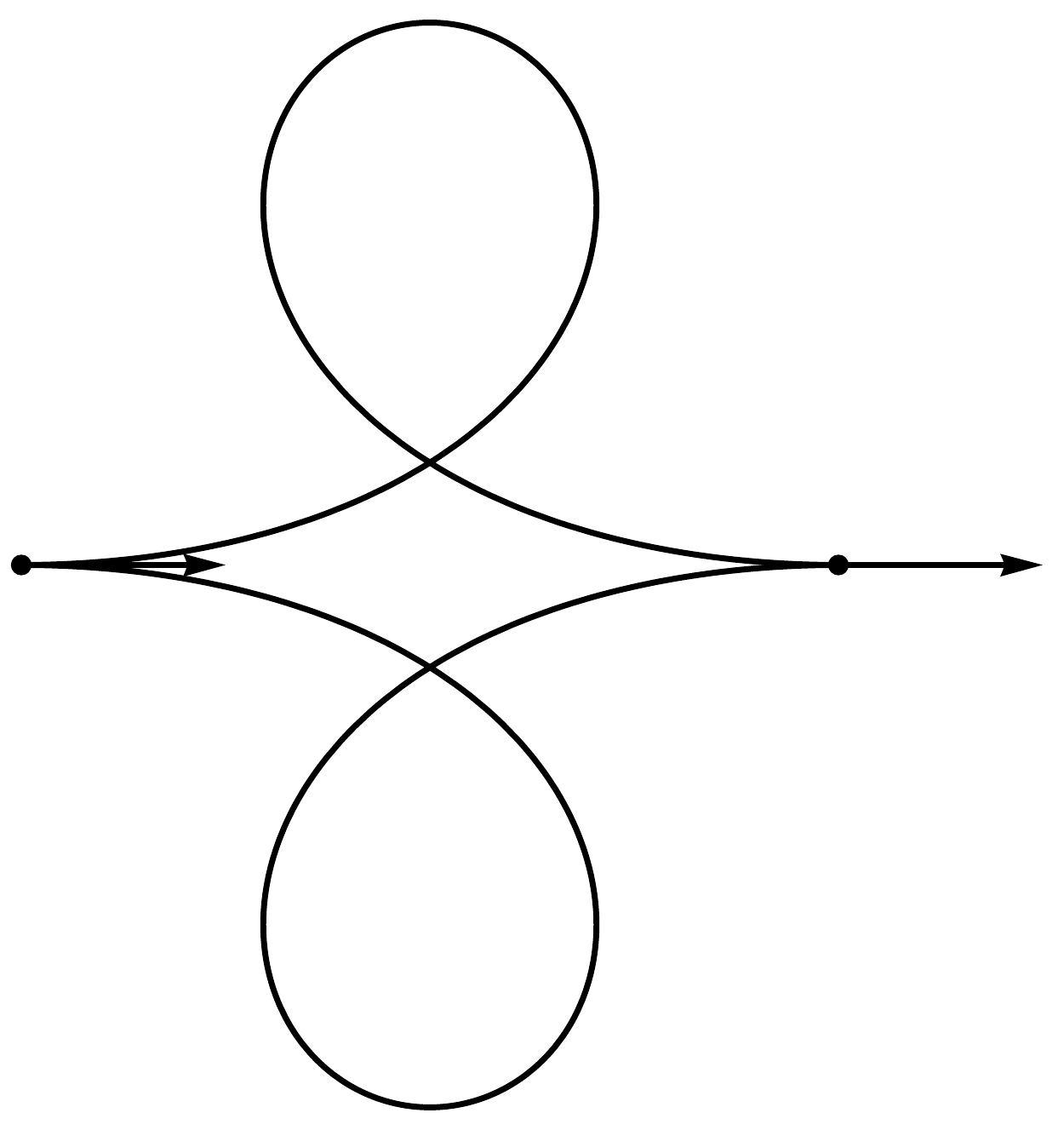}{Optimal elasticae for $(x_1, y_1, \t_1) = (x_1, 0, 0)$, $0 < x_1 < x_*$}{fig:x1g0y10th102opt2}{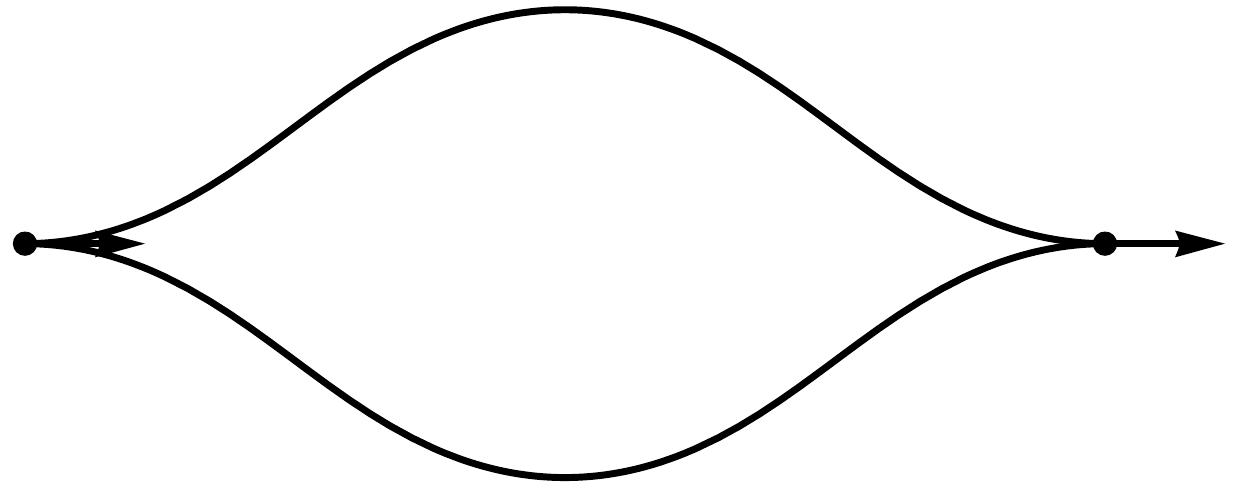}{Optimal elasticae for $(x_1, y_1, \t_1) = (x_1, 0, 0)$, $x_* < x_1 < 1$}{fig:x1g0y10th102opt1}

\onefiglabel{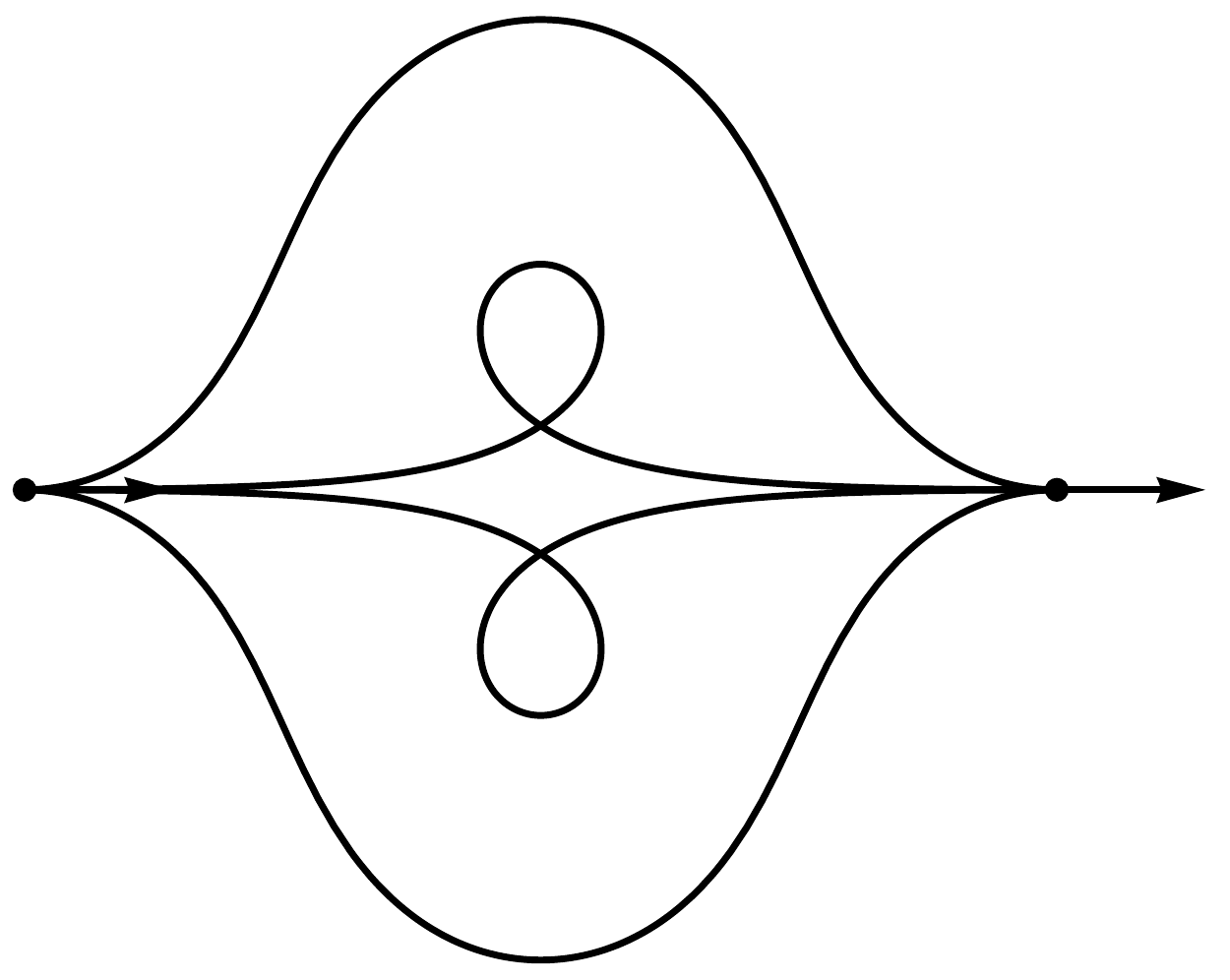}{Optimal elasticae for $(x_1, y_1, \t_1) = (x_*, 0, 0)$}{fig:x1g0y10th102opt4}

 In the case $x_1 < 0$ there are two   optimal non-inflectional elasticae, see Fig.~\ref{fig:x1l0y10th10}.

\onefiglabel{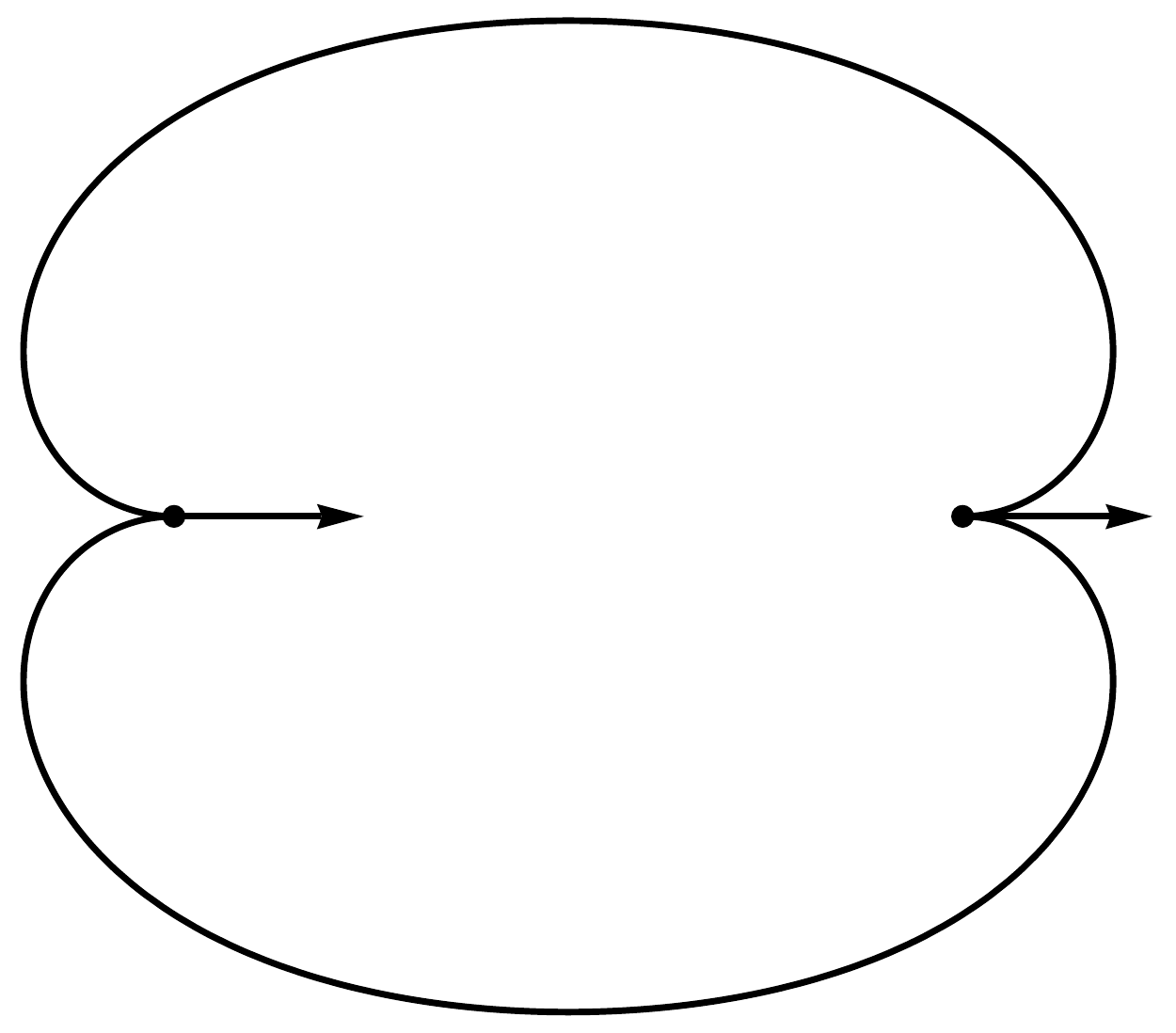}{Optimal elasticae for $(x_1, y_1, \t_1) = (x_1, 0, 0)$, $x_1 < 0$}{fig:x1l0y10th10}

\subsection{The case $x_1 = 1$, $y_1 = 0$, $\t_1 = 0$}
In this case there exists a unique optimal elastica --- the straight line.

\section{Conclusion}\label{sec:concl}
This paper completes our planned study of Euler's elastic problem via geometric control techniques~\cite{notes}. The theoretical analysis describes the structure of optimal solutions and yields effective computation algorithms for numerical evaluation of these solutions for given boundary conditions. We believe that the approach developed in the study of Euler's problem would be useful for other symmetric optimal control problems, e.g. invariant sub-Riemannian problems on 3-D Lie groups~\cite{agrachev_barilari}, nilpotent sub-Riemannian problems~\cite{engel, max3}, problems on rolling sphere~\cite{sphere_roll}, and others.

\end{document}